\documentclass[12pt]{amsart}

\usepackage[english]{babel}

\usepackage[pdftex,textwidth=400pt,marginratio=1:1]{geometry}

\usepackage{amsfonts}

\usepackage[dvips]{graphics}

\usepackage{amsmath}

\usepackage{bbm}

\usepackage{amsthm}

\usepackage{amssymb}
\usepackage{enumerate}

\usepackage{eufrak}

\usepackage{cancel}

\usepackage{color}

\usepackage[curve]{xypic}

\allowdisplaybreaks

\newtheorem{theorem}{Theorem}[section]

\newtheorem{proposition}[theorem]{Proposition}

\newtheorem{proposition-definition}[theorem]{Proposition-Definition}

\newtheorem{lemma}[theorem]{Lemma}

\newtheorem{remark}[theorem]{Remark}

\theoremstyle{definition}

\newtheorem{definition}[theorem]{Definition}

\theoremstyle{property}

\makeatletter
\newcommand{\contraction}[5][1ex]{%
  \mathchoice
    {\contraction@\displaystyle{#2}{#3}{#4}{#5}{#1}}%
    {\contraction@\textstyle{#2}{#3}{#4}{#5}{#1}}%
    {\contraction@\scriptstyle{#2}{#3}{#4}{#5}{#1}}%
    {\contraction@\scriptscriptstyle{#2}{#3}{#4}{#5}{#1}}}%
\newcommand{\contraction@}[6]{%
  \setbox0=\hbox{$#1#2$}%
  \setbox2=\hbox{$#1#3$}%
  \setbox4=\hbox{$#1#4$}%
  \setbox6=\hbox{$#1#5$}%
  \dimen0=\wd2%
  \advance\dimen0 by \wd6%
  \divide\dimen0 by 2%
  \advance\dimen0 by \wd4%
  \vbox{%
    \hbox to 0pt{%
      \kern \wd0%
      \kern 0.5\wd2%
      \contraction@@{\dimen0}{#6}%
      \hss}%
    \vskip 0.5ex
    \vskip\ht2}}

\newcommand{\contraction@@}[3][0.05em]{%
  \hbox{%
    \vrule width #1 height 0pt depth #3%
    \vrule width #2 height 0pt depth #1%
    \vrule width #1 height 0pt depth #3%
    \relax}}
\makeatother

\DeclareFontFamily{OT1}{rsfs}{}
\DeclareFontShape{OT1}{rsfs}{n}{it}{<-> rsfs10}{}
\DeclareMathAlphabet{\curly}{OT1}{rsfs}{n}{it}

\newcommand\mdot{{\scriptscriptstyle\bullet}}

\renewcommand\L{\mathcal L}

\renewcommand\O{\mathcal O}
\newcommand\PP{\mathbb P}

\newcommand\cE{\mathcal E}

\newcommand\F{\mathcal F}

\newcommand\G{\mathcal G}

\newcommand\sfZ{\mathsf Z}

\newcommand\M{\mathcal M}
\newcommand\N{\mathcal N}

\newcommand\C{\mathbb C}

\newcommand\cC{\mathcal C}

\newcommand\cK{\mathcal K}

\newcommand\cQ{\mathcal Q}

\newcommand\R{\mathbb R}

\newcommand\ccR{\mathcal R}

\newcommand\Z{\mathbb Z}

\newcommand\ch{\mathrm{ch}}

\newcommand\Into{\ar@{^(->}[r]<-.3ex>}

\renewcommand\hom{\curly H\!om}

\newcommand\ext{\curly Ext}

\newcommand\Pic{\operatorname{Pic}}

\newcommand\Quot{\operatorname{Quot}}

\newcommand\beq[1]{\begin{equation}\label{#1}}
\newcommand\eeq{\end{equation}}
\newcommand\beqa{\begin{eqnarray*}}
\newcommand\eeqa{\end{eqnarray*}}

\newcommand{\cT}{\mathcal{T}}

\DeclareRobustCommand{\SkipTocEntry}[4]{}

\begin{document}
\title[Stable reflexive sheaves and localization]{Stable reflexive sheaves and localization}
\author[A.~Gholampour and M. Kool]{Amin Gholampour and Martijn Kool}

\maketitle

\begin{abstract}
We study moduli spaces $\N$ of rank 2 stable reflexive sheaves on $\PP^3$. Fixing Chern classes $c_1$, $c_2$, and summing over $c_3$, we consider the generating function $\sfZ^{\mathrm{refl}}(q)$ of Euler characteristics of such moduli spaces. The action of the torus $T$ on $\PP^3$ lifts to $\N$ and we classify all sheaves in $\N^T$. This leads to an explicit expression for $\sfZ^{\mathrm{refl}}(q)$. Since $c_3$ is bounded below and above, $\sfZ^{\mathrm{refl}}(q)$ is a polynomial. We find a simple formula for its leading term when $c_1=-1$. 

Next, we study moduli spaces of rank 2 stable torsion free sheaves on $\PP^3$ and consider the generating function of Euler characteristics of such moduli spaces. We give an expression for this generating function in terms of $\sfZ^{\mathrm{refl}}(q)$ and Euler characteristics of Quot schemes of certain $T$-equivariant reflexive sheaves, which are studied elsewhere. 
Many techniques of this paper apply to any toric 3-fold. In general, $\sfZ^{\mathrm{refl}}(q)$ depends on the choice of polarization which leads to wall-crossing phenomena. We briefly illustrate this in the case of $\PP^2 \times \PP^1$.
\end{abstract}

\thispagestyle{empty}
\renewcommand\contentsname{\vspace{-8mm}}

\section{Introduction}

In general, vector bundles on a complex smooth projective variety $X$ do not form a compact moduli space. Fixing a polarization $H$, we denote by $\M_{X}^{H}(r,c_{\mdot})$ the quasi-projective moduli space of rank $r$ $\mu$-stable\footnote{See \cite{HL} for the definitions of torsion free and stability.} torsion free sheaves on $X$ with total Chern class $c_\mdot$. If rank and degree are coprime (not required in this paper), then Gieseker and $\mu$-stability coincide and there are no strictly semistable sheaves. In this case the moduli space is compact. Although $\M_{X}^{H}(r,c_{\mdot})$ contains vector bundles as an open subset, this set is in general not dense and can be empty. Besides being interesting in their own right, exciting invariants can be extracted from these moduli spaces. When $X$ is a Calabi-Yau or Fano 3-fold, R.~P.~Thomas showed they carry a perfect obstruction theory \cite{Tho}, which can be used to construct deformation invariants of $X$. 

The moduli spaces $\M_{X}^{H}(r,c_{\mdot})$ have been heavily studied in the case $r=1$ or $X=S$ is a surface. We refer the reader to \cite[Part II]{HL} for references. In this paper, we are interested in the case $r=2$ and $X=\PP^3$. The action of the dense open torus $T = \C^{*3}$ on $X$ lifts to $\M_{X}^{H}(2,c_1,c_2,c_3)$ and our goal is to study the fixed point locus. Good knowledge of the fixed point locus might allow one to compute the generating function of topological Euler characteristics $e(\cdot)$
\[
\sfZ_{c_1,c_2}(q) = \sum_{c_3} e(\M_{\PP^3}(2,c_1,c_2,c_3)) q^{c_3}.
\]
\indent In the case $r=1$ or $X=S$ is a surface, generating functions of Euler characteristics have very interesting combinatorial and number-theoretic properties. E.g.~for $r=1$ and $X$ toric, they are generating functions of monomial ideals. For $X=S$ a surface they are (quasi)-modular forms in many cases \cite{Got1, Kly2, Yos, VW, Man1, Man2, Man3}. 

Given a coherent sheaf $\F$ on $X$, its double dual is defined by 
$$\F^{**}:=\hom_X(\hom_X(\F,\O_X),\O_X).$$ 
There is a natural map of coherent sheaves $\theta : \F \rightarrow \F^{**}$, which is an injection if and only if $\F$ is torsion free.  A coherent sheaf $\F$ is called \emph{reflexive} if $\theta$ is an isomorphism.  Locally free sheaves (vector bundles) are examples of reflexive sheaves. If $\F$ is a torsion free sheaf, then its double dual $\F^{**}$, which is automatically reflexive, is called the \emph{reflexive hull} of $\F$. A reflexive sheaf is fully determined by its restriction to the complement of any codimension $\geq 2$ closed subset of $X$ \cite[Prop.~1.6]{Har}. For this reason, reflexive sheaves are much easier to understand than general torsion free sheaves.

 In the case $X=S$ is a surface, reflexive sheaves are the same as locally free sheaves and the cokernel of $\F \rightarrow \F^{**}$ is 0-dimensional. Applying the double-dual map at the level of moduli spaces gives a product formula for the generating function \cite[Prop.~3.1]{Got2}
\[
\sum_{c_2} e(\M_{S}^{H}(r,c_1,c_2)) q^{c_2} = \frac{1}{\prod_{k>0}(1-q^k)^{r e(S)}} \sum_{c_2} e(\N_{S}^{H}(r,c_1,c_2)) q^{c_2},
\]
where $\N_{S}^{H}(r,c_1,c_2)$ is the moduli space of rank $r$ $\mu$-stable vector bundles on $S$ with Chern classes $c_1, c_2$. In the case $r=2$ and $X = \PP^3$ the cokernel of $\F \rightarrow \F^{**}$ has dimension $\leq 1$ and there is no simple product formula. Nevertheless, we still propose to study the (in general non-compact!) moduli spaces of reflexive sheaves $\N_{\PP^3}(2,c_1,c_2,c_3)$ and the generating function
\[
\sfZ_{c_1,c_2}^{\mathrm{refl}}(q) = \sum_{c_3} e(\N_{\PP^3}(2,c_1,c_2,c_3)) q^{c_3}.
\]
The goal of this paper is to fully describe the fixed point locus of $\N_{\PP^3}(2,c_1,c_2,c_3)$ and the generating function $\sfZ_{c_1,c_2}^{\mathrm{refl}}(q)$. Note that after tensoring with a line bundle, we can assume without loss of generality that\footnote{The cohomology ring of $\PP^3$ is $H^{2*}(\PP^3,\Z) \cong \Z[h] / (h^4)$. Therefore each $H^{2i}(\PP^3,\Z)$ is generated by $h^i$ and we can view any Chern class $c_i \in H^{2i}(\PP^3,\Z)$ as an integer.} $c_1 \in \{ -1, 0 \}$. 

In this paper, we classify all \emph{$T$-equivariant} rank 2 $\mu$-stable reflexive sheaves on $\PP^3$ (Proposition \ref{class}). We distinguish three types of sheaves: type $(1)$ (generic) and types (2) and (3) (degenerations of type $(1)$). Type $(1)$ will come with continuous moduli (as $T$-equivariant sheaves), whereas types $(2)$ and $(3)$ have no moduli. From the classification, we get an expression for $\sfZ^{\mathrm{refl}}_{c_1, c_2}(q)$.
\begin{theorem} \label{main}
For any  $c_1, c_2$, there are explicit sets $D_1(c_1,c_2), D_2(c_1,c_2) \subset \Z^4$, and $D_3(c_1,c_2) \subset \Z^3$ defined in Section 3.3, such that 
\[
\sfZ_{c_1,c_2}^{\mathrm{refl}}(q) = -\sum_{{\bf{v}} \in D_1(c_1,c_2)} q^{C_{1}({\bf{v}})} + \sum_{{\bf{v}} \in D_2(c_1,c_2)} 6q^{C_{2}({\bf{v}})} + \sum_{{\bf{v}} \in D_3(c_1,c_2)} 4q^{C_{3}({\bf{v}})},
\]
where $C_{i}({\bf{v}})$ are the following cubic forms
\[
C_{1}({\bf{v}}) = \sum_{1 \leq i<j<k \leq 4} v_i v_j v_k, \ C_{2}({\bf{v}}) = (v_1+v_2)v_3 v_4, \ \mathrm{and} \ C_{3}({\bf{v}}) = v_1 v_2 v_3.
\]
\end{theorem}
The three terms in this formula correspond to the contribution of each of the three types in the classification. For $c_1=-1$ and $c_2=1,2,3$, this gives the following expressions for $\sfZ_{c_1,c_2}^{\mathrm{refl}}(q)$
\[
4 q, \ 24 q^4, \ - 4 q^7 + 36 q^9, \  \ldots
\]
This list suggests $\sfZ_{c_1,c_2}^{\mathrm{refl}}(q)$ is always a \emph{polynomial}. This is indeed the case for any generating function of rank 2 $\mu$-stable reflexive sheaves on any polarized smooth projective 3-fold $(X,P)$. More precisely, the third Chern class $c_3 \in H^6(X,\Z) \cong \Z$ of such a sheaf is bounded below by zero and is bounded above\footnote{This is proved in Proposition \ref{bounds}. The argument for the upper bound was suggested to us by R.~P.~Thomas.} by some constant depending on $X$, $H$, $c_1$, and $c_2$. However, finding an \emph{explicit} value for the upper bound on any given $(X,H)$ is a much harder problem \cite{Har, Ver1, Ver2}. On $X=\PP^3$, R.~Hartshorne proves the following.
\begin{theorem} [Hartshorne] \label{Hartshorne}
Let $\F$ be a rank 2 $\mu$-stable reflexive sheaf on $X = \PP^3$ with Chern classes $c_1, c_2, c_3$. Then the following holds:
\begin{enumerate}
\item[$\mathrm{(i)}$] $c_3 = c_1 c_2 \ \mathrm{mod} \ 2$,
\item[$\mathrm{(ii)}$] if $c_1 \in \{-1,0\}$, then $c_2 > 0$,
\item[$\mathrm{(iii)}$] if $c_1=-1$, then $0 \leq c_3 \leq c_{2}^{2}$ and if $c_1=0$, then $0 \leq c_3 \leq c_{2}^{2} - c_2 +2$. Both upper bounds are sharp.
\end{enumerate}
\end{theorem}
We use our classification to reprove Harthorne's theorem in the $T$-equivariant case (Proposition \ref{equivHartshorne}). 
This gives a direct proof of polynomiality of $\sfZ_{c_1,c_2}^{\mathrm{refl}}(q)$. In addition, we classify all $T$-equivariant rank 2 $\mu$-stable reflexive sheaves with $c_2$ equal to the upper bound. This leads to the following formula
\begin{equation} \label{upperEuler} 
e(\N_{\PP^3}(2,-1,c_2,c_{2}^{2})) = \left\{ \begin{array}{cc} 4 & \mathrm{if} \ c_2=1 \\ 12 c_2 & \mathrm{if} \ c_2>1. \end{array} \right.
\end{equation}

Next, we show how to modify the generating function of Theorem \ref{main} to include torsion free sheaves (Proposition \ref{torsionfree}). To get explicit formulae for the full generating function $\sfZ_{c_1,c_2}(q)$ for torsion free sheaves, one has to compute Euler characteristics of Quot schemes of 0 or 1-dimensional cokernels of certain $T$-equivariant rank 2 $\mu$-stable reflexive sheaves. The components of the fixed loci of these Quot schemes are products of $\PP^1$'s similar to the case of stable pairs on toric 3-folds studied by R.~Pandharipande and R.~P.~Thomas \cite{PT2}. 
This is the topic of another paper \cite{GKY}, which is a joint work with B.~Young.  For low values of $c_2$, we get closed formulae involving the MacMahon function. An example is given by the following theorem. 
\begin{theorem} \cite{GKY} \label{thmGKY}
For $c_1=-1$ and $c_2=1$, $\sfZ_{c_1,c_2}(q)$ is equal to 
\[
4(q + q^{-1}) M(q^{-2})^8,
\] 
where $M(q) = \prod_{k>0} 1 / (1-q^k)^k$ is the MacMahon function.
\end{theorem}

Many techniques of this paper extend to arbitrary smooth toric 3-folds. In the last section, we discuss some of the new features and complications arising in the general case. One such feature is dependence of the generating function on choice of polarization. This leads to wall-crossing phenomena. We illustrate this in the case $X = \PP^2 \times \PP^1$. \\

\noindent \textbf{Notation.} Whenever we write ``for all/there exist $\{i,j,k,l\} = \{1,2,3,4\}$'' we mean ``for all/there exist $i \in \{1,2,3,4\}$, $j \in \{1,2,3,4\} \setminus \{ i \}$, $k \in \{1,2,3,4\} \setminus \{i,j\}$, and $l \in \{1,2,3,4\} \setminus \{i,j,k\}$''. \\
 
\noindent \textbf{Acknowledgments.} We would like to thank R.~P.~Thomas, P.~Vermeire, and B.~Young for very helpful discussions. A.G.~was partially supported by NSF grant DMS-1406788. M.K.~was supported by EPSRC grant EP/G06170X/1, ``Applied derived categories''.

\section{Equivariant sheaves on toric varieties}

This section is a brief exposition of the main results of \cite{Kly1, Kly2, Per, Koo}. We review Klyacho's and Perling's description of $T$-equivariant coherent, torsion free, and reflexive sheaves on toric varieties.

Let $X$ be a smooth toric variety of dimension $d$ with torus $T$. Let $M = X(T)$ be the character group of $T$ (written additively) and denote its dual by $N$. Denote the natural pairing by $\langle \cdot, \cdot \rangle  : M \times N \rightarrow \Z$. Then $N$ is a rank $d$ lattice containing a fan\footnote{We always assume $\Delta$ contains cones of dimension $d$.} $\Delta$ and the data $(N,\Delta)$ completely describes $X$. We refer to Fulton's book \cite{Ful} for the general theory. We recall that there is a bijection between the cones $\sigma \in \Delta$ and the $T$-invariant affine open subsets $U_\sigma \subset X$. \\

\noindent \emph{The affine case}. Suppose $X = U_\sigma$. Let $S_{\sigma} = \{ m \in M \ | \ \langle m , \sigma \rangle \geq 0 \}$. This semi-group gives rise to an algebra $\C[S_\sigma]$, which is exactly the coordinate ring of $U_\sigma$. Therefore, quasi-coherent sheaves on $U_\sigma$ are the same as $\C[S_\sigma]$-modules. More precisely, the global section function gives an equivalence of categories
\[
H^0(\cdot) : \mathrm{Qco}(U_\sigma) \rightarrow \C[S_\sigma]\textrm{-}\mathrm{Mod}.
\]
Under this equivalence, coherent sheaves correspond to finitely generated modules. It will not come as a surprise that this equivalence can be extended to an equivalence between $T$-equivariant quasi-coherent sheaves and $\C[S_\sigma]$-modules with regular $T$-action. The map goes as follows. For a $T$-equivariant quasi-coherent sheaf $(\F,\Phi)$ on $U_\sigma$, use the $T$-equivariant structure $\Phi$ to define a regular $T$-action on $H^0(\F)$. Since $T$ is diagonalizable, a $T$-action on $H^0(\F)$ is equivalent to a decomposition of $H^0(\F)$ into weight spaces
\[
H^0(\F) = \bigoplus_{m \in M} H^0(\F)_m.
\] 
Therefore $T$-equivariant quasi-coherent sheaves on $U_\sigma$ are nothing but $M$-graded $\C[S_\sigma]$-modules, i.e.~there exists an equivalence of categories
\[
H^0(\cdot) : \mathrm{Qco}^T(U_\sigma) \rightarrow \C[S_\sigma]\textrm{-}\mathrm{Mod}^{M\textrm{-graded}}.
\]
See \cite{Kan, Per} for details. \\

\noindent \emph{Repackaging in terms of $\sigma$-families}. Following Perling \cite{Per}, we write the data of an $M$-graded $\C[S_\sigma]$-module in a more explicit way. 
\begin{definition}[Perling]
For each $m,m' \in M$ we write $m \leq_\sigma m'$ when $m' - m \in S_\sigma$. A $\sigma$-family $\hat{F}$ consists of the following data: a collection of complex vector spaces $\{F_m\}_{m \in M}$ and linear maps $\{\chi_{m,m'} : F_m \rightarrow F_{m'}\}_{m \leq_\sigma m'}$ such that:
\begin{enumerate}
\item[(i)] $\chi_{m,m} = \mathrm{id}_{F_m}$,
\item[(ii)] $\chi_{m',m^{\prime \prime}} \circ \chi_{m,m'} = \chi_{m, m^{\prime \prime}}$ for all $m \leq_\sigma m' \leq_\sigma m''$.
\end{enumerate}
A morphism between $\sigma$-families $\hat{F}, \hat{G}$ is a collection $\hat{\phi}$ of linear maps $\{\phi_m : F_m \rightarrow G_m\}_{m \in M}$ commuting with the $\chi$'s. \hfill 
\end{definition}
An $M$-graded module $F =\bigoplus_{m \in M} F_m$ clearly gives rise to a $\sigma$-family. We take $\{F_m\}_{m \in M}$ to be the collection of weight spaces. Moreover, for each $m \leq_\sigma m'$ we have $m' - m \in S_{\sigma} \subset M$, so multiplication by the character $m' - m$ gives a linear map $F_m \rightarrow F_{m'}$. This gives an equivalence of categories \cite[Prop.~5.5]{Per}
\[
\C[S_\sigma]\textrm{-}\mathrm{Mod}^{M\textrm{-graded}} \rightarrow \sigma\textrm{-Families}.
\]
Suppose $\sigma$ is a cone of maximal dimension $d$. Choose an order of its rays $(\rho_1, \ldots, \rho_d)$ and choose a primitive generator $n_i$ for each ray $\rho_i$. By smoothness of $U_\sigma$, this gives a basis $(n_1, \ldots, n_d)$ of the lattice $N$. Denote the dual basis by $(m_1, \ldots, m_d)$. This choice induces an isomorphism $U_\sigma \cong \C^d$. Let $\hat{F}$ be a $\sigma$-family. Writing each $m \in M$ as $m = \sum_i \lambda_i m_i$, we define
\[
F(\lambda_1, \ldots, \lambda_d) := F_m.
\]
Moreover, multiplication by $\chi_{m, m+m_i}$ gives linear maps
\[
\chi_i(\lambda_1, \ldots, \lambda_d):= \chi_{m, m+m_i} : F(\lambda_1, \ldots, \lambda_{d}) \rightarrow  F(\lambda_1, \ldots, \lambda_{i-1},\lambda_i+1,\lambda_{i+1}, \ldots ,\lambda_d)
\]
satisfying the obvious commutativity properties. We note some important facts.
\begin{enumerate}
\item[(i)] Let $\F$ be a $T$-equivariant quasi-coherent sheaf with $\sigma$-family $\hat{F}$. Then $\F$ is coherent if  only if $\hat{F}$ has finitely many homogeneous generators. Such $\sigma$-families are called finite \cite[Def.~5.10]{Per}. 
\item[(ii)] Let $\F$ be a $T$-equivariant coherent sheaf with $\sigma$-family $\hat{F}$. Then $\F$ is torsion free if  only if all maps $\{\chi_{m,m'}\}_{m \leq_\sigma m'}$ are injective. This can be seen by noting that a non-trivial kernel gives rise to a lower dimensional $T$-equivariant subsheaf of $\F$ (e.g.~see \cite[Prop.~2.8]{Koo}).
\end{enumerate}

\noindent \emph{Equivariant torsion free sheaves}. Let $\F$ be a $T$-equivariant coherent sheaf on $X$. Let $\{\sigma_1, \ldots, \sigma_e\}$ be the cones of maximal dimension. Note that $e = e(X)$ is the number of $T$-fixed points of $X$, which is equal to the Euler characteristic of $X$. The subsets $U_{\sigma_i} \cong \C^d$ provide a $T$-invariant affine open cover of $X$ and the restrictions $\F|_{U_{\sigma_i}}$ give us a collection of finite $\sigma$-families $\{\hat{F}^{\sigma_i}\}_{i =1, \ldots, e}$. Conversely, suppose we are given \emph{any} collection of finite $\sigma$-families $\{\hat{F}^{\sigma_i}\}_{i =1, \ldots, e}$. When do these $\sigma$-families glue to a $T$-equivariant coherent sheaf on $X$? In this paper, we are only interested in the torsion free sheaves, so we describe the answer for such sheaves only.
As mentioned above, in the torsion free case all the maps $\chi_{m,m'}^{\sigma_i}$ between the weight spaces are injective. We can assume all these maps are actually inclusions\footnote{The precise statement is this: the category of $T$-equivariant torsion free sheaves on $U_{\sigma_i}$ is equivalent to the category of finite $\sigma_i$-families with all maps $\chi_{m,m'}^{\sigma_i}$ injective, which in turn is equivalent to its full subcategory of finite $\sigma_i$-families with all maps $\chi_{m,m'}^{\sigma_i}$ inclusions.}. 

We describe the gluing conditions. For each $i=1, \ldots, e$, let $(\rho^{(i)}_1, \ldots, \rho^{(i)}_d)$ be an ordering of rays of $\sigma_i$. Fix any two $i,j$, then the intersection $\sigma_i \cap \sigma_j$ is a cone of some dimension $p$. Assume w.l.o.g.~that $\sigma_i \cap \sigma_j$ is spanned by the first $p$ rays among $(\rho^{(i)}_1, \ldots, \rho^{(i)}_d)$ and $(\rho^{(j)}_1, \ldots, \rho^{(j)}_d)$. Then the gluing conditions are
\begin{equation} \label{glue}
F^{\sigma_i}(\lambda_1, \ldots, \lambda_p, \infty, \ldots, \infty) = F^{\sigma_j}(\lambda_1, \ldots, \lambda_p, \infty, \ldots, \infty), \forall \ \lambda_1, \ldots, \lambda_p \in \Z.
\end{equation}
This needs some explanation. For fixed $i$ and $\lambda_1, \ldots, \lambda_p \in \Z$, consider 
\[
\{F^{\sigma_i}(\lambda_1, \ldots, \lambda_p, \mu_{p+1}, \ldots, \mu_d)\}_{\mu_{p+1}, \ldots, \mu_{d} \in \Z}.
\]
Since the $\sigma$-family $\hat{F}^{\sigma_i}$ is finite, these vector spaces stabilize for sufficiently large $\mu_{p+1}, \ldots, \mu_{d}$ and we denote the limit by $F^{\sigma_i}(\lambda_1, \ldots, \lambda_p, \infty, \ldots, \infty)$. Moreover, the vector spaces $F^{\sigma_i}(\lambda_1, \ldots, \lambda_d)$ form a multi-filtration of some limiting finite dimensional vector space $F^{\sigma_i}(\infty, \ldots, \infty)$ of dimension $\mathrm{rk}(\F)$. The idea behind the gluing conditions (\ref{glue}) is the following: the left hand side of (\ref{glue}) is the $\sigma$-family $\hat{F}^{\sigma_i}$ restricted to $U_{\sigma_i} \cap U_{\sigma_{j}}$ and the right hand side is the $\sigma$-family $\hat{F}^{\sigma_j}$ restricted to $U_{\sigma_i} \cap U_{\sigma_{j}}$. This description of $T$-equivariant torsion free sheaves is originally due to Klyachko \cite{Kly2}. We summarize:
\begin{theorem}[Klyachko] \label{Kly}
Let $X$ be a smooth toric variety described by a fan $\Delta$ in a lattice $N$ of dimension $d$. Let $\{\sigma_1, \ldots, \sigma_e\}$ be the cones of maximal dimension. For each $i=1, \ldots, e$, let $(\rho^{(i)}_1, \ldots, \rho^{(i)}_d)$ be an ordering of rays of $\sigma_i$. The category of $T$-equivariant torsion free sheaves on $X$ is equivalent to a category $\cT$ which can be described as follows. The objects of $\cT$ are collections of finite $\sigma$-families $\{\hat{F}^{\sigma_i} \}_{i = 1, \ldots, e}$, with all maps $\chi_{m,m'}^{\sigma_i}$ inclusions, satisfying the following gluing conditions. For any two $i,j$, $\sigma_i \cap \sigma_j$ is a cone of some dimension $p$. Assume w.l.o.g.~that $\sigma_i \cap \sigma_j$ is spanned by the first $p$ rays among both $(\rho^{(i)}_1, \ldots, \rho^{(i)}_d)$ and $(\rho^{(j)}_1, \ldots, \rho^{(j)}_d)$.  Then $\hat{F}^{\sigma_i}$, $\hat{F}^{\sigma_j}$ satisfy\footnote{It should be clear how the gluing conditions read when the rays of $\sigma_i \cap \sigma_j$ do not necessarily correspond to the first $p$ rays of $\sigma_i$ and $\sigma_j$.} (\ref{glue}). The maps of $\cT$ are collections of maps of $\sigma$-families $\{\hat{\phi}^{\sigma_i} : \hat{F}^{\sigma_i} \rightarrow \hat{G}^{\sigma_i} \}_{i = 1, \ldots, e}$ such that for each $i,j$ as above$^7$ 
\[
\phi^{\sigma_i}(\lambda_1, \ldots, \lambda_p, \infty, \ldots, \infty) = \phi^{\sigma_j}(\lambda_1, \ldots, \lambda_p, \infty, \ldots, \infty), \ \forall \ \lambda_1, \ldots, \lambda_p \in \Z.
\]
\end{theorem}

At first glance the description in this theorem does not appear coordinate independent.  However, the \emph{only} choice we made is an ordering of the rays of each cone $\sigma_i$ of maximal dimension. For an extension of this theorem to $T$-equivariant pure sheaves of any dimension, see \cite[Sect.~2]{Koo}. \\

\noindent \emph{Equivariant reflexive sheaves}. Clearly, $T$-equivariant reflexive sheaves on $X$ are $T$-equivariant torsion free, but they have an even simpler description. The reason is that reflexive sheaves on $X$ are fully determined by their behaviour off any codimension $\geq 2$ closed subset of $X$ \cite[Prop.~1.6]{Har}. In particular, a reflexive sheaf on a $T$-invariant affine open subset $U_{\sigma_i} \cong \C^d$ is determined by its behaviour on the complement of the union of all codimension 2 coordinate hyperplanes, i.e.~
\[
(\C \times \C^* \times \cdots \times \C^*) \cup (\C^* \times \C \times \C^* \times \cdots \times \C^*) \cup \cdots \cup (\C^* \times \cdots \times \C^* \times \C).
\]
The restrictions to the components of this union are easy to describe:

Let $\Delta(1)$ be the collection of rays of the fan $\Delta$ of $X$. We introduce a category $\ccR$. Its objects are collections of finite-dimensional complex vector spaces $\{V^\rho(\lambda)\}_{\rho \in \Delta(1), \lambda \in \Z}$ which form flags
\[
\cdots \subset V^\rho(\lambda-1) \subset V^\rho(\lambda) \subset V^\rho(\lambda+1) \subset \cdots.
\] 
We require these flags to satisfy $V^\rho(\lambda) = 0$ for $\lambda \ll 0$ and $V^\rho(\lambda) = V^\rho(\lambda+1)$ for $\lambda \gg 0$. We denote the limiting vector space by $V^\rho(\infty)$. The maps in the category $\ccR$ are the obvious ones: linear maps between the limiting vector spaces preserving the flags. There is a natural fully faithful functor $\ccR \rightarrow \cT$ defined as follows. As before, denote the cones of $\Delta$ of maximal dimension by $\sigma_1, \ldots, \sigma_e$. For each $i=1,\ldots, e$, let $(\rho^{(i)}_1, \ldots, \rho^{(i)}_d)$ be an ordering of rays of $\sigma_i$. Then we map  $\{V^\rho(\lambda)\}_{\rho \in \Delta(1), \lambda \in \Z}$ to the following collection of finite $\sigma$-families
\begin{equation*} 
V^{\sigma_i}(\lambda_1, \ldots, \lambda_d) := V^{\rho^{(i)}_{1}}(\lambda_1) \cap \cdots \cap V^{\rho^{(i)}_{d}}(\lambda_d), \ \forall \lambda_1, \ldots, \lambda_d \in \Z.
\end{equation*}
Under the equivalence of categories of Theorem \ref{Kly}, the $T$-equivariant reflexive sheaves on $X$ correspond to the elements of the image of $\ccR \rightarrow \cT$ \cite{Kly1, Kly2}, \cite[Thm.~5.19]{Per}. Since rank 1 reflexive sheaves are line bundles \cite[Prop.~1.9]{Har}, one can easily see that the $T$-equivariant Picard group $\Pic^T(X)$ is isomorphic to $\Z^{\#\Delta(1)}$. \\

\noindent \emph{The toric variety $\PP^3$}. Most of this paper is devoted to the toric 3-fold $X = \PP^3$. However, most results of this paper do generalize to arbitrary smooth projective toric 3-folds. See Section 5 for a more precise discussion. As a toric 3-fold, $\PP^3$ is described by the lattice $N = \Z^3$ and the fan $\Delta$ consisting of 3-dimensional cones $\sigma_1=\langle e_1, e_2, e_3 \rangle_{\Z_{\geq 0}}$, $\sigma_2=\langle e_2, e_3, -e_1-e_2-e_3 \rangle_{\Z_{\geq 0}}$, $\sigma_3=\langle e_1, e_3, -e_1-e_2-e_3 \rangle_{\Z_{\geq 0}}$, $\sigma_4=\langle e_1, e_2, -e_1-e_2-e_3 \rangle_{\Z_{\geq 0}}$. Here $(e_1,e_2,e_3)$ is the standard basis of $\Z^3$. We denote the rays generated by $e_1,e_2,e_3,-e_1-e_2-e_3$ by $\rho_1, \rho_2, \rho_3, \rho_4$ respectively. 

The description of $T$-equivariant torsion free sheaves on $\PP^3$ is coordinate free up to a choice of ordering of the rays of each cone $\sigma_i$. For definiteness, we choose the following ordering
\begin{equation} \label{order}
\sigma_1: \ (\rho_1,\rho_2, \rho_3), \ \sigma_2: \ (\rho_2,\rho_3, \rho_4), \ \sigma_3: \ (\rho_1,\rho_3, \rho_4), \ \sigma_4: \ (\rho_1,\rho_2, \rho_4).
\end{equation}
A $T$-equivariant torsion free sheaf on $\PP^3$ is described by multi-filtrations as in Theorem \ref{Kly}. A $T$-equivariant reflexive sheaf on $\PP^3$ is described by simply attaching a flag to each of the four rays $\rho_1, \ldots, \rho_4$. Specifically, a $T$-equivariant rank 2 reflexive sheaf $\F$ on $\PP^3$ is specified by a collection of flags $\{V^{\rho_i}(\lambda)\}_{i=1, \ldots, 4}$ of $\C^{\oplus 2}$. As we discussed, the corresponding $\sigma$-families are defined by
\begin{align*} 
F^{\sigma_1}(\lambda_1, \lambda_2, \lambda_3) := V^{\rho_1}(\lambda_1) \cap V^{\rho_{2}}(\lambda_2) \cap V^{\rho_{3}}(\lambda_3), \ldots 
\end{align*}
The flags $\{V^{\rho_i}(\lambda)\}_{i = 1, \ldots, 4}$ can be described by indicating the integers where the vector spaces jump together with the 1-dimensional subspace occurring in each flag. More precisely, for each $i = 1, \ldots, 4$, there exist unique integers $u_i \in \Z$, $v_i \in \Z_{\geq 0}$ and a subspace $p_i \in \mathrm{Gr}(1,2) \cong \PP^1$ such that
\begin{equation}
V^{\rho_{i}}(\lambda) = \left\{\begin{array}{cc}  0 & \mathrm{if \ } \lambda < u_{i} \\ p_i & \mathrm{if \ } u_{i} \leq \lambda < u_{i} + v_i \\ \mathbb{C}^{\oplus 2} & \mathrm{if \ } u_{i} + v_i \leq \lambda. \end{array} \right. \nonumber
\end{equation}
Note that $v_i$ could be zero in which case $p_i$ does not occur. At such places, the flag jumps from $0$ to $\C^{\oplus 2}$. 
\begin{definition} \label{toricdata}
Instead of describing a $T$-equivariant rank 2 reflexive sheaf $\F$ on $\PP^3$ by the flags $\{V^{\rho_i}(\lambda)\}_{i = 1, \ldots, 4}$, we can also describe it by the data $\{(u_i, v_i, p_i)\}_{i=1, \ldots, 4}$ introduced above. We refer to $\{(u_i, v_i, p_i)\}_{i=1, \ldots, 4}$ as \emph{toric data} and abbreviate it by $({\bf{u}}, {\bf{v}}, {\bf{p}})$.  
\end{definition}

\section{Equivariant rank 2 reflexive sheaves on $\PP^3$}

In this section, we classify all $T$-equivariant rank 2 $\mu$-stable reflexive sheaves on $\PP^3$. As an application, we prove Theorem \ref{main} of the introduction, which gives an expression for the generating function
\[
\sfZ_{c_1,c_2}^{\mathrm{refl}}(q) = \sum_{c_3} e(\N_{\PP^3}(2,c_1,c_2,c_3)) q^{c_3}.
\]
For any polarized smooth projective 3-fold $X$, we show $\sfZ_{c_1,c_2}^{\mathrm{refl}}(q)$ is a polynomial (Proposition \ref{bounds}). For $X = \PP^3$ we reprove Hartshorne's inequalities in the $T$-equivariant setting (Theorem \ref{Hartshorne} of the introduction). As a corollary, we determine the leading coefficient of $\sfZ_{c_1,c_2}^{\mathrm{refl}}(q)$ (equation (\ref{upperEuler}) of the introduction).

\subsection{Chern classes}

In this section, we compute the Chern classes of a $T$-equivariant rank 2 reflexive sheaf $\F$ on $\PP^3$. \\

\noindent \emph{$T$-equivariant  line bundles on $\PP^3$}. We start with a short description of $T$-equivariant line bundles on $\PP^3$. As we saw in the previous section, these correspond to flags $\{V^{\rho_i}(\lambda)\}_{i=1, \ldots, 4}$ of $\C$. Therefore, they are fully specified by four integers $u_i$, corresponding to the values of $\lambda$ where $V^{\rho_i}(\lambda)$ jumps from 0 to $\mathbb{C}$. We denote the $T$-equivariant line bundle corresponding to $(u_1,u_2,u_3,u_4)$ by $\L_{(u_1,u_2,u_3,u_4)}$. This provides a group isomorphism
\begin{align*}
&\Z^4 \stackrel{\cong}{\longrightarrow} \mathrm{Pic}^{T}(X), \\
&(u_1,u_2,u_3,u_4) \rightarrow [\L_{(u_1,u_2,u_3,u_4)}].
\end{align*}
Forgetting the $T$-equivariant structure, $\L_{(1,0,0,0)}, \L_{(0,1,0,0)}, \L_{(0,0,1,0)}, \L_{(0,0,0,1)}$ are all isomorphic to $\O(-1)$. Hence
\[
c_1(\L_{(u_1,u_2,u_3,u_4)}) = -u_1-u_2-u_3-u_4.
\]
The kernel of the forgetful map 
\[
\mathrm{Pic}^{T}(X) \rightarrow \mathrm{Pic}(X)
\] 
is the character group $M$ so the above isomorphism descends to 
\begin{equation*} \label{character}
\{u_1 + u_2 + u_3 + u_4 = 0\} \subset \Z^4 \stackrel{\cong}{\longrightarrow} M.
\end{equation*}

Another elementary fact is the following. Let $\F$ be a $T$-equivariant torsion free sheaf on $\PP^3$ with $\sigma$-families $\{\hat{F}^{\sigma_i}\}_{i=1,\ldots,4}$. Then $\G = \F \otimes \L_{(u_1,u_2,u_3,u_4)}$ is a $T$-equivariant torsion free sheaf on $\PP^3$ and its $\sigma$-families are given by
\begin{align*}
G^{\sigma_1}(\lambda_1,\lambda_2,\lambda_3)=F^{\sigma_1}(\lambda_1-u_1,\lambda_2-u_2,\lambda_3-u_3), \\
G^{\sigma_2}(\lambda_1,\lambda_2,\lambda_3)=F^{\sigma_2}(\lambda_1-u_2,\lambda_2-u_3,\lambda_3-u_4), \\
G^{\sigma_3}(\lambda_1,\lambda_2,\lambda_3)=F^{\sigma_3}(\lambda_1-u_1,\lambda_2-u_3,\lambda_3-u_4), \\
G^{\sigma_4}(\lambda_1,\lambda_2,\lambda_3)=F^{\sigma_4}(\lambda_1-u_1,\lambda_2-u_2,\lambda_3-u_4).
\end{align*}
This follows from writing out the $M$-grading of a tensor product of two $M$-graded modules and using the ordering of rays (\ref{order}). More details can be found in \cite[Prop.~4.6]{Koo}. \\

\noindent \emph{$T$-equivariant d\'evissage and Chern classes}. In \cite{Kly2}, Klyachko gives an explicit formula for the Chern character of any $T$-equivariant torsion free sheaf on a smooth projective toric variety. We like to take a slightly different viewpoint and use the following lemma instead. A proof can be found in \cite[Lem~7.6]{GJK}.
\begin{lemma} \label{devissage}
Let $X$ be a smooth toric variety described by a fan $\Delta$ in a lattice $N$ of dimension $d$. Let $(\sigma_1, \ldots, \sigma_e)$ be the cones of maximal dimension. For each $i=1, \ldots, k$, let $(\rho^{(i)}_1, \ldots, \rho^{(i)}_d)$ be a choice of ordered rays of $\sigma_i$. Let $\F$, $\G$ be $T$-equivariant torsion free sheaves on $X$ with $\sigma$-families $\{\hat{F}^{\sigma_i}\}_{i=1, \ldots, e}$, $\{\hat{G}^{\sigma_i}\}_{i=1, \ldots, e}$. If $\mathrm{dim}(F^{\sigma_i}(\vec{\lambda})) = \mathrm{dim}(G^{\sigma_i}(\vec{\lambda}))$ for all $i, \vec{\lambda}$, then $\ch(\F) = \ch(\G)$.  
\end{lemma}
\begin{proposition} \label{Chern}
Let $\F$ be a $T$-equivariant rank 2 reflexive sheaf on $\PP^3$ with associated toric data $({\bf{u}}, {\bf{v}}, {\bf{p}})$. Define abbreviations 
\begin{align*}
u&:=u_1+u_2+u_3+u_4, \\
p_{ij}&:=1-\mathrm{dim}(p_i \cap p_j), \\
p_{ijk}&:=1-\mathrm{dim}(p_i \cap p_j \cap p_k).
\end{align*}
Then
\begin{align*}
&c_1(\F) = - \big(2u + \sum_i v_i \big), \\
&c_2(\F) = \frac{1}{4} c_1(\F)^2 + \frac{1}{2} \sum_{i<j} (2p_{ij} - 1) v_i v_j -\frac{1}{4} \sum_i v_{i}^{2}, \\
&c_3(\F) = \sum_{i<j<k} v_i v_j v_k (p_{ij}+p_{ik}+p_{jk} - 2p_{ijk}).
\end{align*}
\end{proposition}

\begin{proof}
We compute the Chern character $\ch(\F)$. We start with the case all $p_i$ are equal. In this case $\F$ is the direct sum of 
\begin{align*}
\L_1 &:= \L_{(u_1,u_2,u_3,u_4)}, \\
\L_2 &:= \L_{(u_1+v_1,u_2+v_2,u_3+v_3,u_4+v_4)},
\end{align*}
and the formula follows from
\begin{align} \label{3dim}
\begin{split}
\ch(\F) &=e^{-u h}+ e^{-(u+\sum_i v_i)h}, 
\end{split}
\end{align}
where $h$ denotes the hyperplane class. 

For general $p_i$, we construct a $T$-equivariant subsheaf $\G \subset \L_1 \oplus \L_2$ such that $\ch(\G)$ is easy to compute and $\mathrm{dim}(F^{\sigma_i}(\vec{\lambda})) = \mathrm{dim}(G^{\sigma_i}(\vec{\lambda}))$ for all $i, \vec{\lambda}$. The result follows from computing $\ch(\G)$ and applying Lemma \ref{devissage}. We define $\G \subset \L_1 \oplus \L_2$ by the following $\sigma$-families $\{\hat{G}^{\sigma_i}\}_{i=1, \ldots, 4}$
\begin{align*}
G^{\sigma_i}(\vec{\lambda}) = \left\{ \begin{array}{cc} 0 & \mathrm{if} \ \mathrm{dim}(F^{\sigma_i}(\vec{\lambda}) = 0 \\ L_{1}^{\sigma_i}(\vec{\lambda}) \oplus L_{2}^{\sigma_i}(\vec{\lambda}) & \mathrm{otherwise}, \end{array} \right.
\end{align*}
where $\{\hat{L}_{a}^{\sigma_i}\}_{i=1, \ldots, 4}$ are the $\sigma$-families of $\L_a$ for $a=1,2$.

The 3-fold $\PP^3$ contains six 1-dimensional torus invariant lines ($ \cong \PP^1$). Fix one of them, say the one corresponding to $\sigma_1 \cap \sigma_2$. For any integers $k,l,m,n$, we define a 1-dimensional $T$-equivariant sheaf $\mathcal{P}_{klmn}^{(12)}$ supported on the line corresponding to $\sigma_1 \cap \sigma_2$. It is defined by the following $\sigma$-families $\{\hat{P}^{(12), \sigma_i}_{klmn}\}_{i=1, \ldots, 4}$: let $\hat{P}^{(12), \sigma_i}_{klmn} = 0$ when $i \neq 1,2$ and 
\begin{align*}
&P^{(12), \sigma_1}_{klmn}(\lambda_1, \lambda_2, \lambda_3) = \left\{\begin{array}{cc} \C & \mathrm{if} \ \lambda_1 \geq k, \ \lambda_2 = l \ \mathrm{} \ \lambda_3=m \\ 0 & \mathrm{otherwise}, \end{array} \right. \\
&P^{(12), \sigma_2}_{klmn}(\lambda_1, \lambda_2, \lambda_3) = \left\{\begin{array}{cc} \C & \mathrm{if} \ \lambda_1=l, \ \lambda_2=m \ \mathrm{} \ \lambda_3 \geq n \\ 0 & \mathrm{otherwise}. \end{array} \right.
\end{align*}
This glues by equations (\ref{glue}), (\ref{order}). The sheaf $\mathcal{P}_{klmn}^{(12)}$ is just the push-forward of $\O_{\PP^1}$ (with trivial $T$-equivariant structure) tensored by $\L_{(k,l,m,n)}$. Similarly, one can define toric sheaves $\mathcal{P}_{klmn}^{(ij)}$ for all $i<j \in \{1, \ldots, 4\}$. The Chern character of $\mathcal{P}_{klmn}^{(12)}$ (or any $\mathcal{P}_{klmn}^{(ij)}$) is easily computed using the following $T$-equivariant resolution
\[
0 \rightarrow \L_{(k,l+1,m+1,n)} \rightarrow \L_{(k,l+1,m,n)} \oplus \L_{(k,l,m+1,n)} \rightarrow \L_{(k,l,m,n)} \rightarrow \mathcal{P}_{klmn}^{(12)} \rightarrow 0.
\]
Here the first map is $v \mapsto (v,v)$, the second map is $(v,w) \mapsto v-w$ and the third map is the cokernel map. The resolution gives
\begin{align*} 
\ch(\mathcal{P}_{klmn}^{(12)}) &= e^{-(k+l+m+n)h} - 2e^{-(k+l+m+n+1)h} + e^{-(k+l+m+n+2)h} \\
&=h^2 - (k+l+m+n+1) h^3.
\end{align*}

The 3-fold $\PP^3$ contains four torus fixed points. Fix one of them, say the one corresponding to $\sigma_1$. For any integers $k,l,m$, we define a $T$-equivariant 0-dimensional sheaf $\cQ_{klm}^{(1)}$ supported on the torus fixed point corresponding to $\sigma_1$. It is defined by the following $\sigma$-families $\{\hat{Q}^{(1), \sigma_i}_{klm}\}_{i=1, \ldots, 4}$: let $\hat{Q}^{(1), \sigma_i}_{klm} = 0$ when $i \neq 1$ and
\begin{align*}
Q^{(1), \sigma_1}_{klm}(\lambda_1, \lambda_2, \lambda_3) = \left\{\begin{array}{cc} \C & \mathrm{if} \ \lambda_1 = k, \ \lambda_2 = l \ \mathrm{} \ \lambda_3=m \\ 0 & \mathrm{otherwise}. \end{array} \right. 
\end{align*}
The sheaf $\cQ_{klm}^{(1)}$ is just the skyscraper sheaf $\O_p$ of the torus fixed point $p$ (with trivial $T$-equivariant structure) tensored by $\L_{(k,l,m,0)}$. Similarly, one can define toric sheaves $\cQ_{klm}^{(i)}$ for all $i=1, \ldots, 4$. The Chern character of $\cQ_{klm}^{(1)}$ (or any $\cQ_{klm}^{(i)}$) is clearly 
\[
\ch(\cQ_{klm}^{(1)}) = h^3.
\]

It is not hard to write down a toric filtration $\G:=\G^{(N)} \subset \G^{(N-1)} \subset \cdots \subset \G^{(1)} \subset \G^{(0)}:=\L_1 \oplus \L_2$ where each quotient $\G^{(i)} / \G^{(i-1)}$ is one of the sheaves $\mathcal{P}_{klmn}^{(ij)}$ or $\cQ_{klm}^{(i)}$. The formula for $\ch(\G)$ is obtained by subtracting the following contribution from equation (\ref{3dim})
\[
\sum_{i<j} \sum_{a=0}^{v_i-1}\sum_{b=0}^{v_j-1}\big(h^2 - (u+v_k+v_l+1+a+b) h^3 \big)p_{ij} + \sum_{i < j < k} v_i v_j v_k p_{ijk} h^3.
\]
Here the first sum is over all $1 \leq i < j  \leq 4$ and $k<l$ are the remaining two indices among $1,2,3,4$.
\end{proof}

\subsection{Classification}

Let $\F$ be a $T$-equivariant reflexive sheaf on $\PP^3$. In order to test whether $\F$ is $\mu$-stable, it suffices to consider $T$-equivariant subsheaves only. It is easy to see that $\F$ is $\mu$-semistable if  only if $\mu(\G) \leq \mu(\F)$ for all $T$-equivariant subsheaves $\G \subset \F$ with $0 < \mathrm{rk}(\G) < \mathrm{rk}(\F)$. This follows from the fact that the Harder-Narasimhan filtration of $\F$ consists of $T$-equivariant subsheaves. With more work, one can show that $\F$ is $\mu$-stable if  only if $\mu(\G) < \mu(\F)$ for all $T$-equivariant subsheaves $\G \subset \F$ with $0 < \mathrm{rk}(\G) < \mathrm{rk}(\F)$ \cite[Prop.~4.13]{Koo}. Moreover, it suffices to test stability for saturated subsheaves only \cite[Def.~1.1.5]{HL}. In addition, a saturated subsheaf of a reflexive sheaf is reflexive \cite[Lem.~II.1.1.16]{OSS}. Therefore, in the rank 2 case, we only need test stability for $T$-equivariant saturated line subbundles of $\F$.
\begin{proposition} [Classification] \label{class}
Let $\F$ be a $T$-equivariant rank 2 reflexive sheaf on $\PP^3$ with associated toric data $({\bf{u}}, {\bf{v}}, {\bf{p}})$. Then $\F$ is $\mu$-stable if  only if one of the following holds:
\begin{enumerate}
\item[$\mathrm{(1)}$] $0<v_i<v_j + v_k + v_l$ for all $\{ i,j,k,l \}=\{ 1,2,3,4 \}$ and all $p_i$ are mutually distinct\footnote{The notation ``for all/there exist $\{ i,j,k,l \}=\{ 1,2,3,4 \}$'' is explained in the introduction.},
\item[$\mathrm{(2)}$] $v_1,v_2,v_3,v_4>0$, there are $\{ i,j,k,l \}=\{1,2,3,4\}$ such that $v_i + v_j < v_k + v_l$, $v_k < v_i + v_j + v_l$, $v_l < v_i + v_j + v_k$, $p_i = p_j$, and  $p_j, p_k, p_l$ are mutually distinct,
\item[$\mathrm{(3)}$] there are $\{ i,j,k,l \}=\{1,2,3,4\}$ such that $v_i=0$, $v_j, v_k, v_l >0$, $v_j < v_k + v_l$, $v_k < v_j + v_l$, $v_l < v_j + v_k$, and $p_j, p_k, p_l$ are mutually distinct.
\end{enumerate}
\end{proposition}

\begin{definition} 
We refer to $T$-equivariant rank 2 $\mu$-stable reflexive sheaves on $\PP^3$ with toric data satisfying (1), (2), (3) as sheaves of \emph{type} (1), (2), (3) respectively. 
\end{definition}

\begin{proof}
Denote the $\sigma$-families of $\F$ by $\{\hat{F}^{\sigma_i}\}_{i=1, \ldots, 4}$.The sheaf $\F$ has at most four $T$-equivariant saturated line subbundles. They can be described as follows. For each $a=1, \ldots, 4$, define the $T$-equivariant line bundle $\L_a$ by the $\sigma$-families
\[
\{\hat{F}^{\sigma_i} \cap p_a\}_{i=1, \ldots, 4}.
\]
The sheaf $\F$ can only be $T$-equivariantly indecomposable if at least three $v_i$'s are positive and at least three $p_i$'s are mutually distinct. \\

\noindent \emph{Case 1}. All $v_i >0$ and all $p_i$ are mutually distinct. From the description of $\F$ in terms of its toric data $({\bf{u}}, {\bf{v}}, {\bf{p}})$, it is clear that there are four saturated $T$-equivariant line subbundles $\L_a$. Using Proposition \ref{Chern} (or \cite[Prop.~3.20]{Koo}), it is easy to show that 
\begin{align*}
\mu(\L_a) = v_a-\sum_{i}(u_i+v_i), \ \mu(\F) = \frac{1}{2}(v_1 + v_2 + v_3 + v_4)-\sum_{i}(u_i+v_i). 
\end{align*} 
Hence $\F$ is $\mu$-stable if  only if  $v_i<v_j + v_k + v_l$ for all $\{ i,j,k,l \}=\{ 1,2,3,4 \}$. \\

\noindent \emph{Cases 2, 3}. Similar.
\end{proof}

\begin{remark} 
In order to compute the slopes in the proof of Proposition \ref{class}, it suffices to know $c_1(\F)$, $c_1(\L_a)$ only. This greatly simplifies the computations. See \cite[Prop.~3.20]{Koo} for a generalization to torsion free sheaves of any rank on any polarized smooth projective toric variety.
\end{remark}

\subsection{Generating function for reflexive sheaves}

In this section, we prove Theorem \ref{main} of the introduction. The proof follows by combining torus localization, the formula for Chern classes in Proposition \ref{Chern}, and the classification in Proposition \ref{class}. For any integers $c_1, c_2$, and writing the components of a vector ${\bf{v}}$ by $v_i$, we introduce three sets
\begin{align*}
D_1(c_1,c_2) := \Big\{ {\bf{v}} \in \Z^{4}_{>0} \ | \ &c_1 + \sum_{i} v_i = 0 \ \mathrm{mod} \ 2, \ \frac{c_{1}^{2}}{4} + \frac{1}{2} \sum_{i < j} v_i v_j - \frac{1}{4} \sum_i v_{i}^{2} = c_2, \\
&v_i < v_j+v_k+v_l \ \forall \{i,j,k,l\}=\{1,2,3,4\}  \Big\}, \\
D_2(c_1,c_2) := \Big\{ {\bf{v}} \in \Z^{4}_{>0} \ | \ &c_1 + \sum_{i} v_i = 0 \ \mathrm{mod} \ 2, v_1 + v_2 < v_3+v_4, \\ 
&v_3 < v_1+v_2+v_4, \ v_4<v_1+v_2+v_3, \\
& \frac{c_{1}^{2}}{4} - v_1 v_2 + \frac{1}{2} \sum_{i < j} v_i v_j - \frac{1}{4} \sum_i v_{i}^{2} = c_2 \Big\}, \\
D_3(c_1,c_2) := \Big\{ {\bf{v}} \in \Z^{3}_{>0} \ | \ &c_1 + \sum_{i} v_i = 0 \ \mathrm{mod} \ 2, \ \frac{c_{1}^{2}}{4} + \frac{1}{2} \sum_{i < j} v_i v_j - \frac{1}{4} \sum_i v_{i}^{2} = c_2, \\ 
&v_i < v_j+v_k \ \forall \{i,j,k\}=\{1,2,3\} \Big\}. 
\end{align*}
\begin{proof} [Proof of Theorem \ref{main}]
Fix $c_1$, $c_2$, and $c_3$. The action of $T$ on $\PP^3$ lifts to an action on $\N_{\PP^3}(2,c_1,c_2,c_3)$ by \cite[Prop.~4.1]{Koo}. Pointwise, this action is $t \cdot [\F] = [t^* \F]$. Hence $e(\N_{\PP^3}(2,c_1,c_2,c_3)) = e(\N_{\PP^3}(2,c_1,c_2,c_3)^T)$. Take an element $[\F] \in \N_{\PP^3}(2,c_1,c_2,c_3)^T$, then $\F$ admits a $T$-equivariant structure \cite[Prop.~4.4]{Koo} and this $T$-equivariant structure is unique up to tensoring by a character \cite[Prop.~4.5]{Koo}.

Let $\cC$ be the collection of $T$-equivariant isomorphism classes of $T$-equivariant rank 2 $\mu$-stable reflexive sheaves on $\PP^3$ with Chern classes $c_1,c_2,c_3$. To any such sheaf, we associate toric data $({\bf{u}},{\bf{v}},{\bf{p}})$ (Definition \ref{toricdata}). Forgetting the $T$-equivariant structure gives a surjective map $\cC \rightarrow \N_{\PP^3}(2,c_1,c_2,c_3)^T$, but this map is not injective. Specifically, for any $[\F] \in \cC$ and character $0 \neq \chi \in M$, the $T$-equivariant sheaves $\F$ and $\F\otimes \O(\chi)$ are isomorphic but not $T$-equivariantly isomorphic. However, for each $[\F] \in \cC$, there exists a \emph{unique} character $\chi \in M$ such that $\F \otimes \O(\chi)$ has toric data $({\bf{u}},{\bf{v}},{\bf{p}})$ satisfying $u_1=u_2=u_3=0$.  Let $\cC^{\mathrm{slice}} \subset \cC$ be the collection of $T$-equivariant isomorphism classes $[\F]$ for which the toric data $({\bf{u}},{\bf{v}},{\bf{p}})$ satisfies $u_1=u_2=u_3=0$. Then $\cC^{\mathrm{slice}} \rightarrow \N_{\PP^3}(2,c_1,c_2,c_3)^T$ is a bijection. In fact, it is proved in \cite[Thm.~4.15]{Koo} that the set $\cC^{\mathrm{slice}}$ can be naturally made into a coarse moduli space of $T$-equivariant sheaves and the bijection is an \emph{isomorphism of schemes}.

Each connected component of $\cC^{\mathrm{slice}}$ contains sheaves of one of the three types in the classification (Proposition \ref{class}). Sheaves of type (2) and (3) always occur as isolated reduced points.  Sheaves of type (1) occur in a component which is isomorphic to 
\begin{align*}
\{(p_1,p_2,p_3,p_4) \in (\PP^1)^4 \ | \ p_i \ \mathrm{mutually \ distinct}\} /_{(v_1,v_2,v_3,v_4)} \mathrm{SL}(2,\C) \cong \C^* \setminus \{1\},
\end{align*}
where $(v_1, v_2, v_3, v_4)$ is the linearization (see \cite[Sect.~4.4]{Koo} for details).

Fix a type $\mathrm{Y}=(1), (2), (3)$. The Chern character of a type $\mathrm{Y}$ sheaf with toric data $({\bf{u}},{\bf{v}},{\bf{p}})$ is given by Proposition \ref{Chern} and only depends on the integers $u:=\sum_i u_i$ and the $v_i$. We denote the expression for this Chern character by 
\[
\ch^\mathrm{Y}(u,{\bf{v}}).
\] 
Then $c_{1}^{\mathrm{Y}}(u, {\bf{v}})=c_1$ if and only if $2 \ | \ c_1 + \sum_i v_i$ and 
\[
u= - \frac{1}{2}\big(c_1 +\sum_i v_i\big).
\]
Therefore, we can see $u$ and $\ch^\mathrm{Y}(u, {\bf{v}})$ as functions\footnote{Recall that we think of $c_1$ as fixed.} of the $v_i$ 
\[
\ch^{\mathrm{Y}}({\bf{v}}):= \ch^{\mathrm{Y}} \Big(- \frac{1}{2}\big(c_1 +\sum_i v_i \big), {\bf{v}} \Big).
\]
We obtain
\begin{align*}
e(\cC^{\mathrm{slice}}) = &\sum_{{\bf{v}} \in D_1(c_1,c_2) \ \mathrm{s.t.} \ c_{3}^{(1)}({\bf{v}}) = c_3} e(\C^* \setminus \{1\}) + \sum_{{\bf{v}} \in D_2(c_1,c_2) \ \mathrm{s.t.} \ c_{3}^{(2)}({\bf{v}}) = c_3} 6e(\{pt\}) + \\
&\sum_{{\bf{v}} \in D_3(c_1,c_2) \ \mathrm{s.t.} \ c_{3}^{(3)}({\bf{v}}) = c_3} 4e(\{pt\}).
\end{align*}
Here terms 1, 2, and 3 correspond to the contributions of components of type (1), (2), and (3) respectively. The factor $6$ corresponds to all ways of choosing $p_i = p_j$, $i \neq j$. The factor $4$ corresponds to all ways of choosing $v_i=0$. 
\end{proof}

It is true that $\sfZ_{c_1,c_2}(q)$ is always a polynomial, though this is not clear from the expression in Theorem \ref{main}. We will deduce polynomiality from Proposition \ref{bounds} below (which holds for any 3-fold). Alternatively, polynomiality of $\sfZ_{c_1,c_2}(q)$ also follows from Proposition \ref{equivHartshorne} below (Harthorne's bounds in the $T$-equivariant case). In addition, we will obtain a formula for the leading term of $\sfZ_{c_1,c_2}(q)$ (equation \eqref{upperEuler} of the introduction).

The following proposition establishes polynomiality on \emph{any} polarized smooth projective 3-fold. Part (i) is a slight generalization of a result of R.~Hartshorne \cite[Prop.~2.6]{Har}. The proof for part (ii) was suggested to us by R.~P.~Thomas.
\begin{proposition} \label{bounds}
Let $X$ be a smooth projective variety of dimension $n$ with polarization $H$. 
\begin{enumerate}
\item[$\mathrm{(i)}$] Suppose $n=3$. Then\footnote{We assume $X$ is connected so $H^6(X,\Z) \cong \Z$ is generated by the class of a point $pt$. Hence we can view $c_3(\F)$ as an integer.} $c_3(\F) = h^0(\ext^1(\F,\O_X))$ for any rank 2 reflexive sheaf $\F$ on $X$. In particular, $c_3(\F) \geq 0$ and $c_3(\F) = 0$ if and only if $\F$ is locally free.
\item[$\mathrm{(ii)}$] Suppose $n \geq 2$ and fix $r \in \Z_{>0}$, $c_1 \in H^2(X,\Z)$, $\ldots$, $c_{n-1} \in H^{2n-2}(X,\Z)$. Then there exists a universal constant $C \in \Z$ depending on $X$, $H$, $r$, $c_1$, $\ldots$, $c_{n-1}$ such that any rank $r$ $\mu$-stable torsion free sheaf $\F$ on $X$ with Chern classes $c_1(\F) = c_1$, $\ldots$, $c_{n-1}(\F) = c_{n-1}$ satisfies $c_n(\F) \geq C$ if $n$ is even and $c_n(\F) \leq C$ if $n$ is odd.
\end{enumerate}
In particular, the generating function
\[
\sum_{c_3} e(\N_{X}^{H}(2,c_1,c_2,c_3)) q^{c_3}
\]
of Euler characteristics of moduli spaces $\N_{X}^{H}(2,c_1,c_2,c_3)$ of rank 2 $\mu$-stable reflexive sheaves with fixed $c_1, c_2$ on any smooth projective 3-fold $X$ with respect to any polarization $H$ is a polynomial.
\end{proposition}
\begin{proof}
In the case $X = \PP^3$, part (i) is proved in \cite[Prop.~2.6]{Har}. Except for a small modification in the beginning, the proof works on any polarized smooth projective 3-fold $X$. For completeness, we quickly give the argument. Since $\F$ is reflexive and of rank 2, $\F^* \cong \F \otimes (\det \F)^{-1}$ \cite[Prop.~1.10]{Har}. We now compute $c_3(\F^*)$ in two ways: using this formula and using a locally free resolution of $\F$. We observe that for any rank $2$ coherent sheaf $\F$ and line bundle $\L$ on any smooth projective $n$-fold $X$, one has $c_3(\F \otimes \L) = c_3(\F)$. This simply follows from $\ch(\F \otimes \L) = \ch(\F) \ch(\L)$ and relies on $\F$ having rank 2. Therefore $c_3(\F^*) = c_3(\F)$. Next, take a locally free resolution $0 \rightarrow \cE_1 \rightarrow \cE_0 \rightarrow \F \rightarrow 0$ of $\F$. The resolution can be taken of length 1 by \cite[Prop.~1.3]{Har} and the Auslander-Buchsbaum formula. Dualizing and taking Chern classes gives
\begin{equation} \label{res}
c_t(\F^*) = c_{-t}(\F) c_t(\ext^1(\F,\O_X)).
\end{equation}
Since $\F$ is locally free outside a codimension $3$ closed subset \cite[Cor.~1.4]{Har}, there are only finitely many points where $\F$ is not locally free. Hence the sheaf $\ext^1(\F,\O_X)$ is 0-dimensional and supported at these points. Consequently $\ch(\ext^1(\F,\O_X) = h^0(\ext^1(\F,\O_X) . pt$. Formula (\ref{res}) gives 
\[
c_3(\F^*) =2h^0(\ext^1(\F,\O_X) -  c_3(\F)
\] 
and part (i) follows from the two expressions for $c_3(\F^*)$. 

The second part can be seen by contradiction. We prove the case $n$ is odd, the even case is similar. Suppose $\{\F_i\}_{i=1}^{\infty}$ is a sequence of rank $r$ $\mu$-stable torsion free sheaves on $X$ with Chern classes $c_1, \ldots, c_{n-1}$ and strictly monotonously increasing $c_{n}(\F_i)$. By passing to a subsequence, we may assume $c_n(\F_i)$ all have the same value modulo $(n-1)!$ and suppose (without loss of generality) this value is $0$. Let $i \in \Z$ be arbitrary. Pick any point $p_1 \in X$ where $\F_i$ is locally free, and any surjection 
\[
\F_i \rightarrow \O_{p_1},
\]
where $\O_{p_1}$ is the structure sheaf of $p_1$. Then the kernel $\cK$ of this surjection is a rank $r$ $\mu$-stable torsion free sheaf with Chern classes $c_1, \ldots, c_{n-1}$ and 
\[
c_n(\cK) = c_n(\F_i) - (-1)^{n+1}(n-1)! = c_n(\F_i) - (n-1)!.
\]
Repeat the argument for $\cK$ and a point $p_2 \in X$ where $\cK$ is locally free. Continuing in this fashion, after $N_i$ steps we produce an element $[\cK_i]$ of the moduli space $\M_{X}^{H}(r,c_1, \ldots, c_{n-1},0)$ of $\mu$-stable torsion free sheaves on $X$ with the indicated Chern classes. To first order, we can freely move each of the points $p_1, \ldots, p_{N_i}$ in $n$ directions, so the Zariski tangent space at $[\cK_i]$ is of dimension $\geq n N_i$. Letting $i \rightarrow \infty$, we conclude that the moduli space $\M_{X}^{H}(r,c_1, \ldots, c_{n-1},0)$ has Zariski tangent spaces of arbitrarily high dimension, a contradiction.
\end{proof}

\begin{remark}
Note that this proposition does \emph{not} imply that generating functions of Euler characteristics of moduli spaces of rank 2 $\mu$-stable \emph{torsion free} sheaves of fixed $c_1, c_2$ on polarized smooth projective 3-folds are polynomial. This cannot be the case as Theorem \ref{thmGKY} already illustrate. 
\end{remark}

\subsection{Hartshorne's inequalities}

The upper bound of Proposition \ref{bounds} is not explicit. For rank 2 $\mu$-stable reflexive sheaves on $X = \PP^3$, R.~Harthorne gives explicit upper bounds as discussed in the introduction (Theorem \ref{Hartshorne}). We use the classification of Proposition \ref{class} to rederive these bounds in the $T$-equivariant case. Moreover, we classify the $T$-equivariant sheaves attaining the upper bound. Note that tensoring with $\O(l)$ induces an isomorphism \cite[Cor.~2.2]{Har}
\[
\N_{\PP^3}(2,c_1,c_2,c_3) \cong \N_{\PP^3}(2,c_1+2l,c_2+c_1 l +l^2,c_3),
\]
so we do not loose any generality by assuming $c_1=-1$ or $0$.
\begin{proposition} \label{equivHartshorne}
Let $\F$ be a $T$-equivariant rank 2 $\mu$-stable reflexive sheaf on $\PP^3$ with Chern classes $c_1, c_2, c_3$ and associated toric data $({\bf{u}},{\bf{v}},{\bf{p}})$. Then
\begin{enumerate}
\item[$\mathrm{(i)}$] $c_3 = c_1 c_2 \ \mathrm{mod} \ 2$,
\item[$\mathrm{(ii)}$] if $c_1=-1$ or $0$, then $c_2 > 0$,
\item[$\mathrm{(iii)}$] if $c_1=-1$, then $0 \leq c_3 \leq c_{2}^{2}$ and if $c_1=0$, then $0 \leq c_3 \leq c_{2}^{2} - c_2 +2$.
\end{enumerate}
If $c_1=-1$, then $c_3(\F) = c_2(\F)^2$ precisely if
\begin{enumerate}
\item[$\mathrm{(a)}$] there are $\{i,j,k,l\} = \{1,2,3,4\}$ such that $p_i = p_j$ and $p_j$, $p_k$, $p_l$ are mutually distinct, $v_i \geq 1$, $v_j \geq 1$, $v_k = 1$, and $v_i + v_j = v_l$ (in this case $\F$ is of type (2)),
\item[$\mathrm{(b)}$] there are $\{i,j,k,l\} = \{1,2,3,4\}$ such that $v_i=0$, $p_j$, $p_k$, $p_l$ are mutually distinct, $v_j = 1$, and $v_k = v_l \geq 1$ (in this case $\F$ is of type (3)).  
\end{enumerate}
If $c_1=0$, then $c_3(\F) = c_2(\F)^2 - c_2(\F) +2$ precisely if
\begin{enumerate}
\item[$\mathrm{(a)}$] $p_1$, $p_2$, $p_3$, $p_4$ are mutually distinct and $v_1 = v_2 = v_3 = v_4 = 1$ (in this case $\F$ is of type (1)),
\item[$\mathrm{(b)}$] there are $\{i,j,k,l\} = \{1,2,3,4\}$ such that $p_i = p_j$ and $p_j$, $p_k$, $p_l$ are mutually distinct, $v_i=v_j=1$, and $v_k = v_l = 2$ (in this case $\F$ is of type (2)),
\item[$\mathrm{(c)}$] there are $\{i,j,k,l\}=\{1,2,3,4\}$ such that $v_i=0$, $p_j$, $p_k$, $p_l$ are mutually distinct, and $v_j = v_k = v_l = 2$ (in this case $\F$ is of type (3)).  
\end{enumerate}
\end{proposition}
\begin{proof}
We treat the case $c_1=-1$ in detail and indicate how to do the case $c_1=0$ afterwards. First assume $\F$ is of type $(1)$ and denote the corresponding toric data by $({\bf{u}},{\bf{v}},{\bf{p}})$. Let $u:=\sum_i u_i$, then by Proposition \ref{Chern}
\[
2 \ | \ -1 + \sum_i v_i, \ u = -\frac{1}{2}\big(-1 + \sum_i v_i \big).
\]
By the same proposition, the formulae for $c_2$ and $c_3$ are 
\[
c_2 = \frac{1}{4} + \frac{1}{4} \sum_{i} v_i(v_j+v_k+v_l - v_i), \ c_3 = \sum_{i<j<k} v_i v_j v_k.
\]
Here the first sum is over all $1 \leq i \leq 4$ and $j<k<l$ are the remaining three indices among $1,2,3,4$ (this notation is used several times in the proof). Since $2 \ | \ -1 + \sum_i v_i$, either exactly one $v_i$ is even or exactly one $v_i$ is odd. A simple computation modulo 2 shows (i). Since $v_j+v_k+v_l - v_i > 0$ by stability (Proposition \ref{class}), we have $c_2 > 0$, which shows (ii). The lower bound of (iii) is obvious from the formula for $c_3$. Next we prove $c_{2}^{2} - c_3 \geq 0$. In fact, we claim the function $f : \R^4 \rightarrow \R$
\[
f(x_1, x_2, x_3, x_4) = \Bigg(\frac{1}{4} + \frac{1}{4} \sum_i x_i(x_j+x_k+x_l - x_i)\Bigg)^2 - \sum_{i<j<k} x_i x_j x_k
\]
is \emph{positive} on the region
\[
\ccR = \big\{(x_1,x_2,x_3,x_4) \in \R^4 \ | \ \frac{3}{2} \leq x_i + 1 \leq x_j + x_k + x_l \ \forall \{i,j,k,l\}=\{1,2,3,4\}  \big\}.
\]
This becomes clear after the coordinate transformation $\xi_i:=x_i - \frac{1}{2}$. Indeed
\begin{align*}
f(\xi_1,\xi_2,\xi_3,\xi_4) = \ &\frac{1}{16}\Bigg( \sum_i \xi_i (\xi_j+\xi_k+\xi_l-\xi_i) \Bigg)^2 + \frac{1}{4} \sum_i \xi_{i}^{2}(\xi_j +\xi_k+\xi_l-\xi_i)  \\
&+\frac{1}{2} \sum_{i<j<k} \xi_i \xi_j \xi_k + \frac{1}{8} \sum_i \xi_i (\xi_j+\xi_k+\xi_l - \xi_i) + \frac{1}{16}
\end{align*}
is clearly positive on 
\[
\ccR = \big\{ (\xi_1,\xi_2,\xi_3,\xi_4) \in \R^4 \ | \ 0 \leq \xi_i \leq \xi_j + \xi_k + \xi_l \ \forall \{i,j,k,l\}=\{1,2,3,4\}  \big\}.
\]

If $\F$ is of type $(2)$ with $p_i=p_j$, then define $v:=v_i+v_j$. The inequalities for $v,v_k,v_l$ and the formulae for $c_2$ and $c_3$ in terms of $v,v_k,v_l$ are exactly the same as for a type $(3)$ sheaf. Therefore, let $\F$ be of type $(3)$ with $v_i=0$ and assume without loss of generality that $i=4$. Then 
\[
c_2 = \frac{1}{4} + \frac{1}{4} \sum_i v_i(v_j+v_k - v_i), \ c_3 = v_1 v_2 v_3.
\]
Here the sum is over all $1 \leq i \leq 3$ and $j<k$ are the remaining two indices among $1,2,3$. Properties (i), (ii), and the lower bound of (iii) are easily verified as before. For the upper bound of (iii), and parts (a), (b), consider the function $f : \R^3 \rightarrow \R$
\[
f(x_1, x_2,x_3) = \Bigg(\frac{1}{4} + \frac{1}{4} \sum_i x_i(x_j+x_k - x_i)\Bigg)^2 - x_1 x_2 x_3.
\]
Using the coordinate transformation $\xi_i := x_i -1$, it is easy to see that $f$ is non-negative on 
\[
\ccR = \big\{(x_1, x_2, x_3) \in \R^3 \ | \  x_i + 1 \leq x_j + x_k \ \forall \{i,j,k\}=\{1,2,3\}  \big\}.
\]
Moreover, $f$ is zero on $\ccR$ precisely if at least two of $x_1,x_2,x_3$ are equal to $1$. This gives the upper bound of (iii) and parts (a) and (b).

For $c_1=0$, (i), (ii), and  the lower bound of (iii) are analogous. The upper bound of (iii) can be shown using the same substitution but the estimates are harder. First one treats the case all $v_i \geq 2$. The cases where some $v_i < 2$ need separate estimates. Parts (a), (b), and (c) follow from such an analysis.  
\end{proof}

\begin{proof}[Proof of equation (\ref{upperEuler}) in the introduction]
By Proposition \ref{equivHartshorne}, sheaves of type (1) do not contribute. For sheaves of type (2), there are $6 \cdot 2=12$ ways of choosing $p_i = p_j$ and $v_k = 1$. Sheaves of this type satisfy $v_i+v_j = v_l$ and Proposition \ref{Chern} gives 
\[
c_2 = v_i+v_j = v_l.
\] 
Hence, for fixed $c_2$, there are $c_2-1$ choices for the values of $v_i,v_j \geq 1$. We conclude that there are $12(c_2-1)$ sheaves of type (2). For sheaves of type (3), there are $4 \cdot 3=12$ ways of choosing $v_i=0$, $v_j=1$, $v_k = v_l$ when $v_k = v_l >1$ and 4 such choices when $v_k=v_l=1$. The formula of Proposition \ref{Chern} gives
\[
c_2 = v_k = v_l.
\]
Since $v_k, v_l \geq 1$, there are 12 such sheaves when $c_2>1$ and 4 when $c_2=1$.
\end{proof}

\begin{remark}
For $c_1=-1$ and any $c_2>0$, there exist $T$-equivariant rank 2 $\mu$-stable reflexive sheaves $\F$ such that $c_1(\F)=-1$, $c_2(\F)=c_2$, $c_3(\F) = c_{2}^{2}$ (and all such $\F$ are explicitly described in Proposition \ref{equivHartshorne}). However, for $c_1=0$ and any $c_2>0$, the upper bound $c_3 = c_{2}^2 - c_2 +2$ is only attained by $T$-equivariant rank 2 $\mu$-stable reflexive sheaves for $c_2=2, 3$ (and, again, all such $\F$ are explicitly described in Proposition \ref{equivHartshorne}). This remark follows immediately from Propositions \ref{Chern}, \ref{equivHartshorne}.
\end{remark}

\section{Generating function for torsion free sheaves}

We now turn our attention to the full generating function
\[
\sfZ_{c_1}(p,q) = \sum_{c_2, c_3} e(\M_{\PP^3}(2,c_1,c_2,c_3)) p^{c_2}q^{c_3}.
\]
of Euler characteristics of moduli spaces $\M_{\PP^3}(2,c_1,c_2,c_3)$ of rank 2 $\mu$-stable \emph{torsion free} sheaves on $\PP^3$ with Chern classes $c_1$, $c_2$, $c_3$. This time, we also sum over the second Chern class. This allows us to ``compute'' this generating function by stratifying over types of reflexive hulls. More precisely, we express this generating function in terms of generating functions of Euler characteristics of Quot schemes of certain $T$-equivariant reflexive sheaves and the generating function of Theorem \ref{main}.

Let $\mathrm{Y}=(1), (2), (3)$ be one of the types of sheaves in the classification of Proposition \ref{class}. Let $\F_\mathrm{Y}({\bf{v}},\bf{p})$ be a $T$-equivariant rank 2 $\mu$-stable reflexive sheaf on $\PP^3$ of type $\mathrm{Y}$ described by toric data $({\bf{0}},{\bf{v}},{\bf{p}})$. I.e.~we take all \emph{all} $u_i=0$. For any $c_2$, $c_3$, consider the Quot scheme
\[
\Quot(\F_\mathrm{Y}({\bf{v}},{\bf{p}}),c_2,c_3)
\]
parametrizing quotients $\F_\mathrm{Y}({\bf{v}}) \twoheadrightarrow \cQ$, where $\cQ$ has dimension $\leq 1$ and Chern classes $c_2$, $c_3$. Note that $\cQ$ is 1-dimensional if and only if $c_2 < 0$ and 0-dimensional if and only if $c_2=0$. Moreover, if $c_2=0$, then $c_3 \geq 0$. We consider the generating function
\[
\mathsf{Q}_{\mathrm{Y}, {\bf{v}}}(p,q):= \sum_{c_2} \sum_{c_3} e(\Quot(\F_\mathrm{Y}({\bf{v}}, {\bf{p}}),c_2,c_3)) p^{-c_2} q^{-c_3 + c_2 \sum_i v_i}.
\]
The choice of signs and shift in the powers of the formal variables $p$ and $q$ is motivated by the formula of Proposition \ref{torsionfree} below.

\begin{remark} \label{indep}
Since $\F_\mathrm{Y}({\bf{v}}, {\bf{p}})$ is $T$-equivariant, there is a natural action of $T$ on the Quot scheme $\Quot(\F_\mathrm{Y}({\bf{v}}, {\bf{p}}),c_2,c_3)$. In \cite{GKY}, we give a combinatorial description\footnote{There are some similarities with the description of fixed point loci of moduli spaces of stable pairs on toric 3-folds. Their components are also products of $\PP^1$'s \cite{PT2}.} of the fixed point locus $\Quot(\F_\mathrm{Y}({\bf{v}}, {\bf{p}}),c_2,c_3)^T$ and show its components are products of $\PP^1$'s. The combinatorial description is in terms of certain triples of 3D partitions. Although the sheaf $\F_\mathrm{Y}({\bf{v}}, {\bf{p}})$ depends on ${\bf{p}}$, it is easy to see that $e(\Quot(\F_\mathrm{Y}({\bf{v}}, {\bf{p}}),c_2,c_3))$ does \emph{not} depend on ${\bf{p}}$. 
\end{remark}

Using Theorem \ref{main} we derive the following structure formula for the generating function $\sfZ_{c_1}(p,q)$. 
\begin{proposition} \label{torsionfree}
For any  $c_1$, we have 
\begin{align*}
\sfZ_{c_1}(p,q) = &-\sum_{{\bf{v}} \in D_1(c_1)} \mathsf{Q}_{1,{\bf{v}}}(p,q) p^{\frac{c_{1}^{2}}{4}+B_1({\bf{v}})} q^{C_{1}({\bf{v}})} \\
&+ \sum_{{\bf{v}} \in D_2(c_1)} 6 \mathsf{Q}_{2,{\bf{v}}}(p,q) p^{\frac{c_{1}^{2}}{4}+B_2({\bf{v}})} q^{C_{2}({\bf{v}})} \\
&+ \sum_{{\bf{v}} \in D_3(c_1)} 4 \mathsf{Q}_{3,{\bf{v}}}(p,q) p^{\frac{c_{1}^{2}}{4}+B_3({\bf{v}})} q^{C_{3}({\bf{v}})}.
\end{align*}
Here $D_i(c_1) := \bigcup_{c_2} D_i(c_1,c_2)$ and $D_i(c_1,c_2) \subset \Z^4$, $C_{i}({\bf{v}})$ are as in Theorem \ref{main}. Moreover, $B_{i}({\bf{v}})$ are the following quadratic forms
\begin{align*}
B_1({\bf{v}}) &= \frac{1}{2} \sum_{i<j} v_i v_j - \frac{1}{4} \sum_i v_{i}^{2}, \\ 
B_2({\bf{v}}) &= -v_1 v_2 + \frac{1}{2} \sum_{i<j} v_i v_j - \frac{1}{4} \sum_i v_{i}^{2}, \\ 
B_3({\bf{v}}) &= \frac{1}{2} \sum_{i<j} v_i v_j - \frac{1}{4} \sum_i v_{i}^{2}.
\end{align*}
\end{proposition}
\begin{proof}
We first observe that summing the generating function of Theorem \ref{main} over $c_2$ and using the formula for $c_2$ of Proposition \ref{Chern}, immediately gives
\begin{align*}
\sfZ_{c_1}^{\mathrm{refl}}(p,q) = &-\sum_{{\bf{v}} \in D_1(c_1)} p^{\frac{c_{1}^{2}}{4}+B_1({\bf{v}})} q^{C_{1}({\bf{v}})} + \sum_{{\bf{v}} \in D_2(c_1)} 6 p^{\frac{c_{1}^{2}}{4}+B_2({\bf{v}})} q^{C_{2}({\bf{v}})} \\
&+ \sum_{{\bf{v}} \in D_3(c_1)} 4 p^{\frac{c_{1}^{2}}{4}+B_3({\bf{v}})} q^{C_{3}({\bf{v}})}.
\end{align*}
Next, we fix $c_1, c_2, c_3$ and consider the double dual map\footnote{The double dual map is \emph{not} a morphism. This is because reflexive hulls of the fibres of a flat family need not form a flat family. However, using a result of J.~Koll\'ar \cite{Kol}, the domain can be written as a disjoint union of locally closed subschemes $S_i$, such that on each $S_i$ the double dual map is a morphism. This is enough for our purposes since we only consider Euler characteristics.}
\[
(\cdot)^{**} : \M_{\PP^3}(2,c_1,c_2,c_3) \longrightarrow \coprod_{c'_2, c'_3} \N_{\PP^3}(2,c_1,c'_2,c'_3).
\]
The fibre over $[\F] \in \N_{\PP^3}(2,c_1,c'_2,c'_3)$ equals (at the level of reduced schemes) the Quot scheme
\[
\Quot(\F,c''_2,c''_3),
\]
where
\[
c'_2 = c_2 + c''_2, \ c'_3 = c_3 + c''_3 + c_1 c''_2.  
\]
Next, we consider the double dual map at the level of fixed point loci
\[
(\cdot)^{**} : \M_{\PP^3}(2,c_1,c_2,c_3)^T \longrightarrow \coprod_{c'_2, c'_3} \N_{\PP^3}(2,c_1,c'_2,c'_3)^T.
\]
Take a closed point $[\F] \in \N_{\PP^3}(2,c_1,c'_2,c'_3)^T$. By the proof of Theorem \ref{main}, we can take $\F$ to be $T$-equivariant and described by toric data $({\bf{u}},{\bf{v}},{\bf{p}})$ with $u_1=u_2=u_3=0$. As we have seen in the proof of Theorem \ref{main}, such a choice of $T$-equivariant structure is unique. Since $\F$ has first Chern class $c_1$, we have
\[
u:=\sum_i u_i = u_4 = -\frac{1}{2} \big(c_1 + \sum_i v_i \big)
\]
by Proposition \ref{Chern}. The sheaf $\F$ is of type $\mathrm{Y}$ for some $\mathrm{Y}=(1), (2), (3)$ by Proposition \ref{class}. Define
\[
\F_\mathrm{Y}({\bf{v}},{\bf{p}}):= \F \otimes \L_{(0,0,0,-u)}.
\]
The line bundles $\L_{(u_1, u_2, u_3, u_4)}$ and the effect of tensoring a $T$-equivariant sheaf with such a line bundle were described in Section 3.1.

In the proof of Theorem \ref{main}, we saw that sheaves of type $(2), (3)$ correspond to isolated reduced points and sheaves of type $(1)$ occur in a connected component isomorphic to $\C^* \setminus \{1\}$. Even though sheaves of type $(1)$ are not isolated, the fixed locus $e(\Quot(\F,c''_2,c''_3))$ is independent of $[\F]  \in \C^* \setminus \{1\}$ by Remark \ref{indep}. Therefore, for each connected component $C \subset \N_{\PP^3}(2,c_1,c'_2,c'_3)^T$, we define $e(\Quot(C,c''_2,c''_3)):=e(\Quot(\F,c''_2,c''_3))$ for any $[\F] \in C$. We conclude
\[
e(\M_{\PP^3}(2,c_1,c_2,c_3)) = \sum_{{\scriptsize{\begin{array}{c} c'_2 = c_2 + c''_2 \\ c'_3 = c_3 + c''_3 + c_1 c''_2 \end{array}}}} \!\!\!\!\! \sum_{{\scriptsize{\begin{array}{c} C \subset \N_{\PP^3}(2,c_1,c'_2,c'_3)^T \\ \mathrm{conn. \ comp.} \end{array}}}} \!\!\!\!\! e(C) e(\Quot(C,c''_2,c''_3)).
\]
The result follows from Theorem \ref{main} and the definitions.
\end{proof}

\begin{remark}
Write $\mathsf{Q}_{\mathrm{Y}, {\bf{v}}}(p,q) = \mathsf{Q}_{\mathrm{Y}, {\bf{v}},0}(q)p^0 + \cdots$. The description of fixed point loci in \cite{GKY} can be used to compute a closed expression for $\mathsf{Q}_{\mathrm{Y}, {\bf{v}},0}(q)$. Proposition \ref{torsionfree} and \cite[Cor.~4.10]{GKY} give Theorem \ref{thmGKY} of the introduction.  
\end{remark}

\section{Other toric 3-folds and wall-crossing}

\noindent \emph{Extension to arbitrary toric 3-folds.} Many of the techniques of this paper readily extend to \emph{any} smooth projective toric 3-fold $X$ with polarization $H$. The fan $\Delta$ defining $X$ determines the categories of $T$-equivariant torsion free and reflexive sheaves on $X$ (Theorem \ref{Kly} and Section 2). The $T$-equivariant rank 2 reflexive sheaves on $X$ are described by toric data $({\bf{u}}, {\bf{v}}, {\bf{p}}) = \{(u_i,v_i,p_i)\}_{i=1, \ldots, l}$, where $l$ is the number of rays of $\Delta$. After computing $H^{2*}(X,\Z)$ explicitly \cite[Sect.~5.2]{Ful}, one can find an expression for the Chern classes $c_i(\F)$ of any $T$-equivariant rank 2 reflexive sheaf $\F$ on $X$ as in Proposition \ref{Chern}. The proof goes exactly the same except that the cohomology ring is different. Since the notion of $\mu$-stability for $T$-equivariant reflexive sheaves is worked out explicitly for any rank and on any toric variety in \cite[Prop.~3.20, 4.13]{Koo}, it is not hard to derive the analog of the classification of Proposition \ref{class} on any given $X$. For example, a type $(1)$ sheaf is described by toric data $({\bf{u}},{\bf{v}},{\bf{p}})$, where $p_1,\ldots, p_l$ are $l$ mutually distinct points on $\PP^1$ and the $v_i$ satisfy
\[
(H^2 \cdot D_i) \ v_i < \sum_{{\scriptsize{\begin{array}{c} j=1 \\ j \neq i \end{array}}}}^{l} (H^2 \cdot D_j) \ v_j, \ \forall i=1, \ldots, l,
\]
where $D_1, \ldots, D_l$ are the toric divisors. Besides type $(1)$, many types of degenerations are possible corresponding to $p_i$'s coming together or $v_i$'s becoming zero like in Proposition \ref{class}. The formula for the Chern character together with the classification gives a generating function
\[
\sfZ_{X,H,c_1,c_2}^{\mathrm{refl}}(q) = \sum_{c_3} e(\N_{X}^{H}(2,c_1,c_2,c_3)) q^{c_3},
\]
as in Theorem \ref{main}. By Proposition \ref{bounds}, this expression is a polynomial but finding sharp bounds like in Proposition \ref{equivHartshorne} seems a hard problem in general. \\

\noindent \emph{Wall-crossing.} It is interesting to study the dependence of $\sfZ_{X,H,c_1,c_2}^{\mathrm{refl}}(q)$ on the choice of polarization $H$. This leads to wall-crossing phenomena as we now illustrate for the case $X = \PP^2 \times \PP^1$. As a toric variety, $\PP^2 \times \PP^1$ is described by the lattice $N = \Z^3$ and the fan $\Delta$ consisting of 3-dimensional cones $\sigma_1 = \langle e_1, e_2, e_3 \rangle_{\Z_{\geq 0}}$, $\sigma_2 = \langle e_2, -e_1-e_2, e_3 \rangle_{\Z_{\geq 0}}$, $\sigma_3 = \langle -e_1 - e_2, e_1, e_3 \rangle_{\Z_{\geq 0}}$, $\sigma_4 = \langle e_1, e_2, -e_3 \rangle_{\Z_{\geq 0}}$, $\sigma_5 = \langle e_2, -e_1-e_2, -e_3 \rangle_{\Z_{\geq 0}}$, and $\sigma_6 = \langle -e_1 - e_2, e_1, -e_3 \rangle_{\Z_{\geq 0}}$. Here $(e_1,e_2,e_3)$ is the standard basis of $\Z^3$. We denote the rays corresponding to $e_1$, $e_2$, $-e_1-e_2$ by $\rho_1$, $\rho_2$, $\rho_3$ and the rays corresponding to $e_3$, $-e_3$ by $\rho_{1}^{\prime}$, $\rho_{2}^{\prime}$. The cohomology ring $H^{2*}(\PP^2 \times \PP^1,\Z)$ is the $\Z$-algebra generated by two elements $h$, $h'$ modulo the relations
\[
h^{\prime 2} = h^{3} = 0.
\]
In particular, the degree two part is generated by $l:=h^2$ and $l' := h h'$. Moreover, $pt := h^2 h'$ is (Poincar\'e dual to) the class of a point. Like in the case of $\PP^3$, we can describe a $T$-equivariant rank 2 reflexive sheaf on $\PP^2 \times \PP^1$ by toric data (cf.~Definition \ref{toricdata}). To the rays $\rho_i$ we associate $u_i, v_i, p_i$ and to the rays $\rho_{i}^{\prime }$ we associate $u_{i'}, v_{i'}, p_{i'}$, so a $T$-equivariant rank 2 reflexive sheaf $\F$ on $\PP^2 \times \PP^1$ is described by toric data 
\[
\{(u_i,v_i,p_i), (u_{i'}, v_{i'}, p_{i'})\}_{i=1,2,3, i'=1,2}.
\]
The Chern classes of such a sheaf are given by
\begin{align*}
c_1(\F) = &- \big(2u+\sum_i v_i \big) h - \big( 2u'+\sum_{i'} v_{i'} \big) h', \\
c_2(\F) = &\frac{1}{4} c_1(\F)^2 + \Big( \frac{1}{2} \sum_{i<j} (2p_{ij} - 1) v_i v_j  - \frac{1}{4} \sum_i v_{i}^{2} \Big) l \\
&+ \Big( \frac{1}{2} \sum_{i} \sum_{i'} (2p_{i i'} - 1) v_i v_{i'} \Big) l, \\
c_3(\F) = &\sum_{i<j} \sum_{i'} v_i v_j v_{i'} (p_{ij}+p_{ii'}+p_{ji'} - 2p_{iji'}) pt,
\end{align*}
where 
\begin{align*}
&u :=\sum_i u_i, \ u':=\sum_{i'} u_{i'}, \\
&p_{ij} :=1-\mathrm{dim}(p_i \cap p_j), \ p_{ii'}:=1-\mathrm{dim}(p_i \cap p_{i'}), \ p_{iji'} :=1-\mathrm{dim}(p_i \cap p_j \cap p_{i'}).
\end{align*}
Analyzing $\mu$-stability as described in the previous paragraph gives rise to a classification of $T$-equivariant rank 2 $\mu$-stable reflexive sheaves on $\PP^2 \times \PP^1$. 
Type $(1)$ sheaves are parametrized by five distinct points on $\PP^1$ and there are 6 degenerations corresponding to how the various points can come together or disappear much like in the case of $\PP^3$. It is not hard to write down an expression for $\sfZ_{X,H,c_1,c_2}^{\mathrm{refl}}(q)$, but the general formula is not very enlightening. 

We denote the polarization by $H = \alpha h + \alpha' h'$, where $\alpha, \alpha' \in \Z_{>0}$. The notion of $\mu$-stability only depends on the ratio $\tau:=\frac{\alpha}{2 \alpha'} > 0$.
We fix $c_1 = a h + a' h'$. Without loss of generality, we take $a,a' \in \{0,1\}$. Rank and degree are coprime if and only if $\alpha a'$ is odd. If this is the case, Gieseker and $\mu$-stability coincide and there are no strictly semistables \cite[Lem.~1.2.13, 1.2.14]{HL}. Since we are mostly interested in this case, we only consider $c_1 = h'$ and $c_1 = h+h'$. 

For $c_1 = h'$ and $c_2 = l'$, we computed $\sfZ_{X,H,c_1,c_2}^{\mathrm{refl}}(q)$ \emph{numerically}\footnote{For each term in the exact expression of $\sfZ_{X,H,c_1,c_2}^{\mathrm{refl}}(q)$, we let $v_i$, $v_{i'}$ run from 0 to 10.} for many values of $\tau$. The experiments suggest the following chamber structure
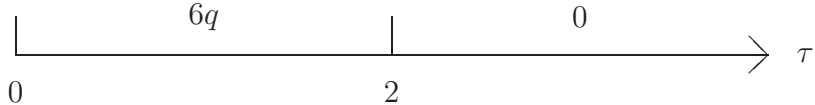
\begin{displaymath}
\xy
(0,0)*{} ; (100,0)*{} **\dir{-} ;
(0,5)*{} ; (0,0)*{} **\dir{-} ;
(50,5)*{} ; (50,0)*{} **\dir{-} ;
(97.5,2.5)*{} ; (100,0)*{} **\dir{-} ; 
(97.5,-2.5)*{} ; (100,0)*{} **\dir{-} ;
(0,-5)*{0} ; (50,-5)*{2} ;
(25,5)*{6q} ; (75,5)*{0} ; (105,0)*{\tau}
\endxy
\end{displaymath}
Repeating the experiment for $c_1 = h+h'$ and $c_2 = 2l'$ suggests the following chamber structure for $\sfZ_{X,H,c_1,c_2}^{\mathrm{refl}}(q)$
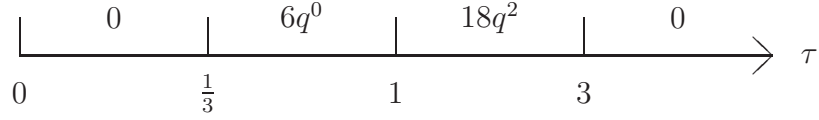
\begin{displaymath}
\xy
(0,0)*{} ; (100,0)*{} **\dir{-} ;
(0,5)*{} ; (0,0)*{} **\dir{-} ;
(25,5)*{} ; (25,0)*{} **\dir{-} ;
(50,5)*{} ; (50,0)*{} **\dir{-} ;
(75,5)*{} ; (75,0)*{} **\dir{-} ;
(97.5,2.5)*{} ; (100,0)*{} **\dir{-} ; 
(97.5,-2.5)*{} ; (100,0)*{} **\dir{-} ;
(0,-5)*{0} ; (25,-5)*{\frac{1}{3}} ; (50,-5)*{1} ; (75,-5)*{3} ;
(12.5,5)*{0} ; (37.5,5)*{6q^0} ; (62.5,5)*{18q^2} ; (87.5,5)*{0} ; (105,0)*{\tau}
\endxy
\end{displaymath}
Here the term $6q^0$ comes from six $T$-equivariant rank 2 $\mu$-stable \emph{locally free} sheaves on $\PP^2 \times \PP^1$ corresponding to isolated fixed points of the moduli space. These sheaves are locally free because their third Chern class is zero (Proposition \ref{bounds}). It should be stressed that these chamber structures follow from extensive Maple experiments, but are not proved. The proof would involve handling a cumbersome system of explicit polynomial equalities and inequalities depending on $\tau$.

\noindent {\tt{amingh@math.umd.edu, \tt{m.kool1@uu.nl}
\end{document}